\documentclass[12pt]{article}
\usepackage[top=2cm,bottom=2cm,left=2cm,right=2cm]{geometry}
\usepackage[dvipdfm]{graphicx}
\usepackage{amsmath,amsthm,amssymb}
\usepackage{mathrsfs}
\usepackage{enumerate}
\usepackage{fancyhdr,url,soul,ulem,color}

\def\u{{\bf u}}

\def\f{{\bf f}}
\def\a{\alpha}
\def\b{\beta}

\def\d{\delta}
\def\D{D}
\def\g{\gamma}
\def\G{\Gamma}
\def\k{\kappa}
\def\l{\lambda}

\def\r{\rho}
\def\s{\sigma}

\def\R{\mathbb R}
\def\N{\mathbb N}
\def\p{\partial}

\def\ep{\varepsilon}

\def\Sm{S_{-\infty}}
\def\Sp{S_\infty}

\numberwithin{equation}{section}
\newtheorem{thm}{Theorem}

\newtheorem{lem}[thm]{Lemma}
\newtheorem{prop}[thm]{Proposition}

\newcommand{\vect}[1]{\mathbf{#1}}

\begin{document}

\title{Traveling Wave Phenomena in a Kermack-McKendrick SIR model}
\author{Haiyan Wang\footnote{Email: Haiyan.Wang@asu.edu}\\
School of Mathematical and Natural Sciences\\ Arizona State University, Phoenix, AZ 85069-7100\\
Xiang-Sheng Wang\footnote{Email: xswang@semo.edu} \\
Department of Mathematics\\ Southeast Missouri State University, Cape Girardeau, MO 63701}
\date{}

\maketitle

\begin{abstract}
We study the existence and nonexistence of traveling waves of general diffusive Kermack-McKendrick SIR models with standard incidence where the total population is not constant.  The three classes,  susceptible $S$, infected $I$ and removed $R$, are all involved in the traveling wave solutions.
We show that the minimum speed for the existence of traveling waves for this three-dimensional non-monotonic system can be derived from its linearizaion at the initial disease-free equilibrium
The proof in this paper is based on Schauder fixed point theorem and Laplace transform and provides a promising method to deal with high dimensional epidemic models.
\end{abstract}



\section{Kermack-McKendrick SIR model with standard incidence }

Compartmental epidemic models describing the transmission of communicable diseases have been extensively studied due to various outbreaks of infectious diseases and many applications to emerging fields such as socio-biological systems, social media. see. e.g. \cite{Brauer2012,Epstein}. To fully understand epidemic models has inspired many unsolved mathematical questions and re-energized novel mathematical research. The simple Kermack-McKendrik model \cite{Kermack1927} is the starting point for many epidemic models.  In a closed population,  which consists of susceptible individuals ($S(t)$), infected individuals ($I (t)$)  and removed individuals ($R(t)$), the simple deterministic susceptible-infected-removed (SIR) is
\begin{eqnarray}
  S'&=&-\tilde{\b} SI,\label{S-ODEN}\\
  I'&=&\tilde{\b} SI-\g I,\label{I-ODEN}\\
  R'&=&\g { I},\label{R-ODEN}
\end{eqnarray}
where $\tilde{\b}$ is the transmission coefficient, $\g$ is the recovery rate.  It is clear that total population $N=S+I+R$ is constant.

It is reasonable to assume that the population is closed and fixed  when  modeling  epidemics  where  the  disease  spreads  quickly in
the population and  dies  out within a short time.  However, if  the  human  or  animal  population  growth  or  decrease  is  significant  or  the
disease  causes  enough  deaths  to  influence  the  population  size, then  it  is  not reasonable  to  assume  that  the  population size is  constant (Mena-Lorca and  Hethcote  \cite{MenaHethcote1992}). To account for variable population sizes, many researchers have proposed epidemic models with transmission coefficient taking the following form
\begin{eqnarray}
  \tilde{\b}(N)&=&\frac{C(N)}{N},\label{S-ODEN1}
\end{eqnarray}
where $N=S+I+R$ is the total population size and $C(N)$ is the adequate contact rate. Then mass-action incidence  corresponds to the choice $C(N)=\beta N$ and standard incidence corresponds to $C(N)=\beta$, where $\beta$ is a positive constant. In general, $C(N)$ is an non-decreasing function with respect to $N$. For example,  $C(N)=\frac{aN}{1+bN+\sqrt{1+2bN}}$ (Heesterbeek,  and   Metz \cite{Heesterbeek1993}), $C(N)=\lambda N^{\alpha}, \alpha=0.05$ (Mena-Lorca and Hethcote \cite{MenaHethcote1992}). Other types of $C(N)$ can be found in \cite{Brauer2012,Thieme} and references therein.

In this paper we are interested in diffusive SIR models where individuals move randomly. While there is a rich literature on spatial epidemic models with delays or integral forms (see \cite{WWW12} for a brief review), several previous works have studied traveling wave solutions of basic diffusive Kermack-McKendrick SIR models. K\"all\'en \cite{Ka84} and Hosono and Ilyas \cite{HI94}  considered the existence of a traveling wave with mass-action incidence.
  \begin{eqnarray}
  \p_tS&=&d_1\p_{xx}S-\b SI ;\label{S-PDEN3}\\
  \p_tI&=&d_2\p_{xx}I+\b SI- \g I.\label{I-PDEN3}
\end{eqnarray}
In particular,  with the aid of the shooting technique and invariant manifold theory developed by Dunbar \cite{Du83,Du84}, Hosono and Ilyas \cite{HI95} proved that
that if the basic reproduction number $\b S_{-\infty}/\g>1$,
then for each $c\ge c^*=2\sqrt{d_2(\b S_{-\infty}-\g)}$
system
(\ref{S-PDEN3}-\ref{I-PDEN3}) has a traveling wave solution $(S(x + ct), I (x + ct))$  satisfying $ S(-\infty)=S_{-\infty}, I(\pm\infty)=0, S(\infty)=S_{\infty}< S_{-\infty}.$  On the other hand, there is no traveling solution for (\ref{S-PDEN3}-\ref{I-PDEN3}) if ${\b S_{-\infty}/\g}\le1$. In a more recent work  \cite{WWW12}, X-S. Wang, H. Wang and Wu proved a similar result for  the following diffusive Kermack-McKendrick SIR model with standard incidence
  \begin{eqnarray}
  \p_tS&=&d_1\p_{xx}S-\frac{\b SI}{S+I} ;\label{S-PDEN4}\\
  \p_tI&=&d_2\p_{xx}I+\frac{\b SI}{S+I}- \g I.\label{I-PDEN4}
\end{eqnarray}
(\ref{S-PDEN4}-\ref{I-PDEN4}) reflects that the recovereds are removed from the population and thus not involved in the contact and disease transmission (\cite{Brauer2012}).  The total population in (\ref{S-PDEN4}-\ref{I-PDEN4}) becomes $S+I$ and is not constant.  Related works can be found in a brief review at the end of \cite{WWW12}.

In this paper, we shall examine more general diffusive Kermack-McKendrick SIR models with the assumption that some of the infective individuals will be removed from the population due to disease-induced death or quarantine, but the recovered individuals will return in the community. Thus, the total population is $N=S+I+R$. Because of mobility of individuals it is more plausible to assume that the total population is not fixed. For simplicity, we will concentrate on standard incidence rate, $C(N)=\beta$ and study the diffusive Kermack-McKendrick SIR model when individuals move randomly, which is given by the following reaction-diffusion system
  \begin{eqnarray}
  \p_tS&=&d_1\p_{xx}S-\frac{\b SI}{N},\label{S-PDEN}\\
  \p_tI&=&d_2\p_{xx}I+\frac{\b SI}{N}-\g I -\delta I,\label{I-PDEN}\\
  \p_tR&=&d_3\p_{xx}R+\g { I},\label{R-PDEN}
\end{eqnarray}
here $N=S+I+R$ is the total population at location $x$ and time $t$, $d_1$, $d_2$ and $d_3$ are the diffusion rates of the susceptible, infective and recovered individuals, respectively. $\g\ge0$ is the recovery rate and $\delta \geq 0$ is the death/quanrantine rate of infective individuals.  The model captures the essential transmission dynamics with standard incidence and predicts infection propagation from the initial source of an outbreak.

As the model focuses on the outbreak situation and ignores the natural and death process, the model system (\ref{S-PDEN}-\ref{R-PDEN}) has infinitely many disease-free equilibria $(S, 0, R)$ with arbitrary $S\ge 0, R \geq 0$.
If we consider the corresponding spatial-homogenous ordinary differential system and linearize it around the trivial disease-free equilibrium $(S_{-\infty}, 0, 0)$, we obtain a simple linear equation for the infective individuals:
$$I'(t)=\b I-(\g+\delta)I.$$
A standard application of next-generation method \cite{DHR90,vW02} gives an explicit formula for the basic reproduction number
$$R_0={\b\over\g+\delta}.$$
As we shall see later, the basic reproduction number is an important threshold parameter in the existence theorem of traveling wave solutions.

\section{Traveling Wave Solutions}
A traveling wave solution is a special type of solutions with the form $(S(x+ct), I(x+ct, t),R(x+ct))$, and represents the transition process of an outbreak from the initial disease-free equilibrium $(S(-\infty), 0, R(-\infty))$ to another disease-free state $(S(\infty),0,R(\infty))$ with $S(\infty)$ being determined by the transmission rate and the disease specific recovery rate, as well as possibly the mobility of individuals. For applications to the disease control and prevention, it is important to determine whether traveling waves exist and what the minimal wave speed $c$ is. Thus we shall look for traveling wave solutions of the form $(S(x+ct), I(x+ct), R(x+ct))$.
\begin{align}
  cS'&=d_1S''-{\b SI\over S+I+R};\label{S-ODE}\\
  cI'&=d_2I''+{\b SI\over S+I+R}-(\g+\d) I;\label{I-ODE}\\
  cR'&=d_3R''+\g I.\label{R-ODE}
\end{align}
It is known that for cooperative systems, the the minimal wave speed can be determined from linearization at low population densities \cite{Weinberger2002,Weinberger2002-1,Weinberger2007}.  Building on the prior work,  one of the authors \cite{Wang2010jns} showed that, for cooperative and a large class of non-cooperative systems, the speed of traveling wave solutions are simply the eigenvalues of the parameterized Jacobian matrix of its linearized system at the initial state and therefore the so-called minimum speed is the simply the minimum of the eigenvalues. Analogous formula for recursion systems was developed in Lui \cite{Lui1989}, and  Weinberger, Lewis and  Li \cite{Weinberger2002-1} extended it to cooperative systems of reaction-diffusion equations based on the time $1$ maps. Here we follow the direct derivation in \cite{Wang2010jns} from a traveling wave solution. Let us consider a system of reaction-diffusion equations
\begin{equation}\label{eq11}
\u_t=D\u_{xx}+\f(\u) \text{ for } x \in \mathbb{R},\; t\geq 0
\end{equation}
where $\u=(u_i)$, $D=\text{diag} (d_1, d_2, ...,d_N),  d_i>0 \text{ for } i=1,...,N$
$$\f(u)=(f_1(u),f_2(u),...,f_N(u)),$$
We are looking for a traveling wave solution $u$ of (\ref{eq11}) of the form $\u=\u(x+ct), u \in C(\mathbb{R}, \mathbb{R}^N)$ with a speed of $c$ . Substituting $\u(x,t)=\u(x+ct)$ into (\ref{eq11}) and letting $\xi=x+ct$, we obtain the wave equation
\begin{equation}\label{eq211}
D\u''(\xi)-c\u'(\xi)+\f(\u(\xi))=0\text{ for } \xi \in \mathbb{R}.
\end{equation}
Now if we look for a solution of the form
$(u_i)=\big ( e^{\lambda \xi}\eta^i_{\lambda}\big),\lambda>0, \eta_{\lambda}=(\eta^i_{\lambda})>>0$ for the linearization of (\ref{eq211}) at an initial equilibrium $E$,
we arrive at the following system
\begin{equation*}
\text{diag}(d_i \lambda^2 -c \lambda)\eta_{\lambda}+f'(E)\eta_{\lambda}=0
\end{equation*}
which can be rewritten as the following eigenvalue problem
\begin{equation}\label{egenvalue}
\frac{1}{\lambda}A_{\lambda}\vect{\eta_{\lambda}}= c \vect{\eta_{\lambda}},
\end{equation}
where
\begin{equation*}
A_{\lambda}=(a^{i,j}_{\lambda})=\text{diag}(d_i\lambda^2)+\f'(E)
\end{equation*}
 Let $\Psi(A_{\lambda})$ be the spectral radius of $A_{\lambda}$ for $ \lambda \in [0, \infty)$,  $$
\Phi(\lambda)=\frac{1}{\lambda} \Psi(A_{\lambda})> 0.
$$

In \cite{Wang2010jns}, under assumption that $\f'(E)$ has nonnegative off diagonal elements and others conditions, it was shown that $\Phi(\lambda)$ is a convex-like function and $\Phi(\lambda)$ goes to $\infty$ at both of $0$ and $\infty$. Therefore $\Phi(\lambda)$ assumes the minimum over the domain $(0, \infty)$, which is the minimum speed of (\ref{eq11}).
$$
c^*=\inf_{\lambda>0}\Phi(\lambda)>0
$$
For  the SIR model (\ref{S-PDEN}-\ref{R-PDEN}), $\f$ is no longer cooperative and some of the off diagonal elements of $\f'(E)$ may be negative.  It is still an open question what additional conditions would make $\Phi(\lambda)$ maintain the convex-like properties.

Nevertheless, we can calculate the minimum wave speed of (\ref{S-PDEN}-\ref{R-PDEN}) from its largest eigenvalue of linearized system at the initial equilibrium $E=(S_{-\infty}, 0, 0)$.  From biological perspective, we are interested in a traveling wave solution connecting $(S_{-\infty}, 0, 0)$  to another disease-free state $(S_\infty,0,R_\infty)$. Now a standard procedure can calculate that  the Jacobian of (\ref{S-PDEN}-\ref{R-PDEN}) at $(S_{-\infty}, 0, 0)$ is $$
\f'(E)=\left(
  \begin{array}{ccc}
    0 & -\beta &0\\
    0 & \beta-\gamma-\delta & 0\\
    0& \gamma & 0
  \end{array}
\right).
$$
For $\lambda \geq 0$, three eigenvalues of the matrix  $$
\left(
  \begin{array}{ccc}
    d_1 \lambda^2 & -\beta &0\\
    0 & d_2 \lambda^2+\beta-\gamma-\delta & 0\\
    0& \gamma & d_3 \lambda^2
  \end{array}
\right)
$$
are $d_1\lambda^2, d_2\lambda^2+\beta-\gamma-\delta, d_3 \lambda^2$. For $\lambda=0$, the largest eigenvalue is $\beta-\gamma-\delta$. Therefore, the minimum wave speed $c^*$ is the minimum of $$
\inf_{\lambda>0}\frac{d_2\lambda^2+\beta-\gamma-\delta}{\lambda}.
$$
A standard calculation shows that $c^*=2\sqrt{d_2(\b-\g-\d)}$.  The following theorem confirms that $c>c^*:=2\sqrt{d_2(\b-\g-\d)}$ is the cut-off number for the existence of traveling solutions connecting the disease-free equilibrium $(S_{-\infty}, 0, 0)$ to another disease-free state $(S_\infty,0, R_{\infty})$.  The two facts indicate that the wave speed of (\ref{S-PDEN}-\ref{R-PDEN}) can be determined from its linearization at  the initial disease-free equilibrium. Our main result is stated as follows.

\begin{thm}\label{thm} Assume that the constants $d_i>0$ with $i=1,2,3$, $\beta>0$, $\gamma>0$ and $\delta \geq 0.$ if
  \begin{align}\label{d3}
  d_3<2d_2,
  \end{align}
  then the minimal wave speed of (\ref{S-PDEN}-\ref{R-PDEN}) can be determined from its linearizaion at the initial disease-free equilibrium.  Specifically, for any $\Sm>0$, $R_0:=\b/(\g+\d)>1$ and $c>c^*:=2\sqrt{d_2(\b-\g-\d)}$, there exist $\Sp<\Sm$ and a traveling wave solution  for (\ref{S-PDEN}-\ref{R-PDEN}) such that
  $S(-\infty)=\Sm$, $S(\infty)=\Sp$, $I(\pm\infty)=0$, $R(-\infty)=0$ and $R(\infty)=\g(\Sm-\Sp)/(\g+\d)$.

  Furthermore, $S(x)$ is decreasing, $0\le I(x)\le\Sm-\Sp$ for $x\in\R$, $R(x)$ is increasing, and
  \begin{align}
    \int_{-\infty}^\infty(\g+\d) I(x)dx=\int_{-\infty}^\infty {\b S(x)I(x)\over S(x)+I(x)+R(x)}dx=c(\Sm-\Sp).
  \end{align}
  On the other hand, if $c<c^*$ or $R_0\le1$, then there does not exist a non-trivial and non-negative traveling wave solution for (\ref{S-PDEN}-\ref{R-PDEN}) such that
  $S(-\infty)=\Sm$, $S(\infty)<\Sm$, $I(\pm\infty)=0$ and $R(-\infty)=0$.
\end{thm}

The technical condition (\ref{d3}) is  similar to the first inequality of assumption (2.7) in \cite{Weinberger2002}. It will be used in the construction of super-solutions and sub-solutions. The result also reveals that the basic reproduction number $R_0$ plays an essential role as in the corresponding spatial-homogenous ordinary differential system.

\section{Existence of traveling solutions}
Throughout this section, we assume that $R_0:=\b/(\g+\d)>1$, $c>c^*:=2\sqrt{d_2(\b-\g-\d)}$ and
the inequality (\ref{d3}) is satisfied.
It is noted that $R_0$ is the basic reproduction number for the ordinary differential system without diffusion \cite{DHR90,vW02}.
Moreover, linearizing the equation for $I$ at the point $(\Sm,0,0)$ gives the characteristic function
\begin{align}\label{f}
  f(\l):=-d_2\l^2+c\l-(\b-\g-\d).
\end{align}
We denote by
\begin{align}\label{l0}
  \l_0:={c-\sqrt{c^2-4d_2(\b-\g-\d)}\over2d_2}>0
\end{align}
the smaller positive root of the characteristic function $f(\l)$.
It is readily seen that the inequality (\ref{d3}) implies
\begin{align}\label{d3-ine}
  c-d_3\l_0>0.
\end{align}
Let $\a_1$, $\a_2$ and $\a_3$ be three sufficiently large constants, we define the second-order differential operator $\D_i$ with $i=1,2,3$ by
\begin{align}\label{Di}
  \D_ih:=-d_ih''+ch'+\a_ih
\end{align}
for any $h\in C^2(\R)$.
Let
\begin{align}\label{li}
  \l_i^\pm={c\pm\sqrt{c^2+4d_i\a_i}\over2d_i} \;\;\;\;\; (\text{note that } \l_i^{-}<0<-\l_i^{-}<\l_i^{+})
\end{align}
be the two roots of the function
\begin{align}\label{fi}
  f_i(\l):=-d_i\l^2+c\l+\a_i.
\end{align}
Denote
\begin{align}\label{ri}
  \r_i:=d_i(\l_i^+-\l_i^-)=\sqrt{c^2+4d_i\a_i}.
\end{align}
The inverse operator $\D_i^{-1}$ is given by the following integral representation
\begin{align}\label{Di-}
  (\D_i^{-1}h)(x):={1\over\r_i}\int_{-\infty}^xe^{\l_i^-(x-y)}h(y)dy+{1\over\r_i}\int_x^\infty e^{\l_i^+(x-y)}h(y)dy
\end{align}
for $h\in C_{\mu^-,\mu^+}(\R)$ with $\mu^->\l_i^-$ and $\mu^+<\l_i^+$, where
\begin{align}\label{Cmu}
  C_{\mu^-,\mu^+}(\R):=\{h\in C(\R):\ \sup_{x\le0}|h(x)e^{-\mu^-x}|+\sup_{x\ge0}|h(x)e^{-\mu^+x}|<\infty\}.
\end{align}
It is readily seen from its integral representation in (\ref{Di-}) that $\D_i^{-1}h$ is differentiable and
\begin{align}
  (\D_i^{-1}h)'(x)&={\l_i^-\over\r_i}\int_{-\infty}^xe^{\l_i^-(x-y)}h(y)dy+{\l_i^+\over\r_i}\int_x^\infty e^{\l_i^+(x-y)}h(y)dy;\label{Di-'}\\
  (\D_i^{-1}h)''(x)&={(\l_i^-)^2\over\r_i}\int_{-\infty}^xe^{\l_i^-(x-y)}h(y)dy+{(\l_i^+)^2\over\r_i}\int_x^\infty e^{\l_i^+(x-y)}h(y)dy-{h(x)\over d_i}.\label{Di-''}
\end{align}
We choose
$$\a_1>\b,\quad \a_2>\g+\d\quad \mbox{and}\quad \a_3>0$$
be sufficiently large such that $|\l_i^-|=-\l_i^->\l_0>0$ for $i=1,2,3$.
Given $\mu>\l_0>0$ such that $\mu<-\l_i^-$ for all $i=1,2,3$,  we have
$$
\l_0<\mu <- \l_i^-<\l_i^+,\ i=1,2,3
$$
(see the definitions of $\l_0$ and $\l_i^-$ in (\ref{l0}) and (\ref{li})) and
$$
\l_i^-< -\mu < \mu<\l_i^+,\ i=1,2,3.
$$
Now we can define the Banach space
\begin{align}\label{Bmu}
  B_\mu(\R,\R^n):=\underbrace{C_{-\mu,\mu}(\R)\times \cdots\times C_{-\mu,\mu}(\R)}_{n}
\end{align}
equipped with the norm
\begin{align}\label{norm}
  |u|_\mu:=\max_{1\le i\le n}\sup_{x\in\R}e^{-\mu|x|}|u_i(x)|,
\end{align}
where $u=(u_1,\cdots,u_n)\in B_\mu(\R,\R^n)$ with $n$ being a positive integer.
We then define a map $F=(F_1,F_2,F_3)$ on the space $B_\mu(\R,\R^3)$: given $u=(u_1,u_2,u_3)\in B_\mu(\R,\R^3)$, let
\begin{align}
  F_1(u_1,u_2,u_3)&:=\D_1^{-1}[\a_1u_1-\b u_1u_2/(u_1+u_2+u_3)];\label{F1}\\
  F_2(u_1,u_2,u_3)&:=\D_2^{-1}[\a_2u_2+\b u_1u_2/(u_1+u_2+u_3)-(\g+\d) u_2];\label{F2}\\
  F_3(u_1,u_2,u_3)&:=\D_3^{-1}[\a_3u_3+\g u_2].\label{F3}
\end{align}
The following lemma shows that the fixed point of the map $F$ is indeed a traveling wave solution.
\begin{lem}
  Let $(S,I,R)\in B_\mu(\R,\R^3)$ be a fixed point of the map $F$, then $(S,I,R)$ satisfies the traveling wave equations (\ref{S-ODE}-\ref{R-ODE}).
\end{lem}
\begin{proof}
  Set $h_1:=\a_1S-\b SI/(S+I+R)$. It follows from (\ref{ri}), (\ref{Di-}), (\ref{Di-'}), (\ref{Di-''}) and the fact that $\l_1^\pm$ are the roots of $f_1(\l)=-d_1\l^2+c+\a_1$ that
  $$-d_1(\D_1^{-1}h_1)''+c(\D_1^{-1}h_1)'+\a_1(\D_1^{-1}h_1)=h_1.$$
  Since $(S,I,R)$ is a fixed point of $F$, it follows that $\D_1^{-1}h_1=S$. Thus, the above equation is the same as (\ref{S-ODE}). Similarly, we can show that the other two equations (\ref{I-ODE}) and (\ref{R-ODE}) are also satisfied.
\end{proof}
For $x\in\R$, we define super-solutions and sub-solutions as follows:
\begin{align}
  S_+(x)&:=S_{-\infty};\label{S+}\\
  S_-(x)&:=\max\{S_{-\infty}(1-M_1e^{\ep_1x}), 0\};\label{S-}\\
  I_+(x)&:=e^{\l_0x};\label{I+}\\
  I_-(x)&:=\max\{e^{\l_0x}(1-M_2e^{\ep_2x}), 0\};\label{I-}\\
  R_+(x)&:={\g\over c\l_0-d_3\l_0^2}e^{\l_0x};\label{R+}\\
  R_-(x)&:=\max\{{\g\over c\l_0-d_3\l_0^2}e^{\l_0x}(1-M_3e^{\ep_3x}), 0\},\label{R-}
\end{align}
where $M_1,M_2,M_3,\ep_1,\ep_2,\ep_3$ are six positive constants to be determined in the following lemma. Its proof can be found in Appendix.
\begin{lem}\label{lem-sub-ine}
  Given sufficiently large $M_1>0$, $M_2>0$, $M_3>0$ and sufficiently small $\ep_1>0$, $\ep_2>0$, $\ep_3>0$, we have
  \begin{equation}\label{S-ine}
    -\b I_+\ge-d_1S_-''+cS_-'
  \end{equation}
  for $x\le x_1:=-\ep_1^{-1}\ln M_1$, and
  \begin{equation}\label{I-ine}
    {\b S_-I_-\over S_-+I_++R_+}-(\g+\d) I_-\ge-d_2I_-''+cI_-'
  \end{equation}
  for $x\le x_2:=-\ep_2^{-1}\ln M_2$, and
  \begin{equation}\label{R-ine}
    \g I_-\ge-d_3R_-''+cR_-'
  \end{equation}
  for $x\le x_3:=-\ep_3^{-1}\ln M_3$.
\end{lem}

With the aid of the super-solutions and sub-solutions, we are now ready to define a convex set $\G$ as
\begin{equation}\label{G}
  \G:=\{(S,I,R)\in B_\mu(\R,\R^3):\ S_-\le S\le S_+~\&~I_-\le I\le I_+\&~R_-\le R\le R_+\}.
\end{equation}
Since $\mu>\l_0>0$, it is easily seen that $\G$ is uniformly bounded with respect to the norm $|\cdot|_\mu$ defined in (\ref{norm}).
To prove invariance of the convex set $\G$ under the map $F$, we shall make use of the following results which were also proved in \cite{WWW12}. For completeness, its proof can be found in Appendix.
\begin{lem}[\cite{WWW12}]\label{lem-D-D}
  Let $i=1,2,3$. We have
  \begin{align}\label{D-D}
    \D_i^{-1}(\D_ih)=h
  \end{align}
  for any $h\in C^2(\R)$ such that $h, h', h''\in C_{\mu^-,\mu^+}(\R)$ with $\mu^->\l_i^-$ and $\mu^+<\l_i^+$.
  Let
  $$g(x):=\max\{e^{\l x}(1-Me^{\ep x}), 0\}$$
  for some $M>0$, $\ep>0$ and $\l$ such that $\l_i^-<\l<\l+\ep<\l_i^+$, we have
  \begin{align}\label{D-D-ine}
    \D_i^{-1}(\D_ig)\ge g.
  \end{align}
  Here $\D_ig$ is understood as a piecewise defined function:
  \begin{align*}
    (\D_ig)(x)=\begin{cases}
      f_i(\l)e^{\l x}-Mf_i(\l+\ep)e^{(\l+\ep)x},\ & x<x^*:=-\ep^{-1}\ln M;\\
      0,\ & x>x^*:=-\ep^{-1}\ln M.
    \end{cases}
  \end{align*}
\end{lem}

Now we are ready to show that the convex set $\G$ defined in (\ref{G}) is invariant under the map $F=(F_1,F_2,F_3)$ defined in (\ref{F1}), (\ref{F2}) and (\ref{F3}). Its proof can be found in Appendix.
\begin{lem}\label{lem-invariant}
  The operator $F=(F_1,F_2,F_3)$ maps $\G$ into $\G$, namely, for any $(S,I,R)\in B_\mu(\R,\R^3)$ such that $S_-\le S\le S_+$, $I_-\le I\le I_+$ and $R_-\le R\le R_+$, we have
  $$S_-\le F_1(S,I,R)\le S_+,$$
  and
  $$I_-\le F_2(S,I,R)\le I_+,$$
  and
  $$R_-\le F_3(S,I,R)\le R_+.$$
\end{lem}

Before applying Schauder fixed point theorem, we shall verify that $F$ is continuous and compact on $\G$ with respect to the norm $|\cdot|_\mu$ defined in (\ref{norm}). The proof is standard and can be found in Appendix.
\begin{lem}\label{lem-cc}
  The map $F=(F_1,F_2,F_3):\G\to\G$ defined in (\ref{F1}), (\ref{F2}) and (\ref{F3}) is continuous and compact with respect to the norm $|\cdot|_\mu$ defined in (\ref{norm}).
\end{lem}
The following proposition gives the first part of our main theorem.
\begin{prop}
  The map $F$ has a fixed point $(S,I,R)\in\G$ which satisfies the equations (\ref{S-ODE}-\ref{R-ODE}).
  As $x\to-\infty$, we have
  \begin{align}\label{asymp-}
  &S(x)\to S_{-\infty},~I(x)\sim e^{\l_0 x},\ R(x)\sim{\g e^{\l_0x}\over c\l_0-d_3\l_0^2},\nonumber\\ &S'(x),~ I'(x),~ R'(x),~S''(x),~ I''(x),~R''(x)\to0.
  \end{align}
  As $x\to\infty$, we have
  \begin{align}\label{asymp+}
  &S(x)\to S_\infty<S_{-\infty},~I(x)\to0,~R(x)\to{\g(\Sm-\Sp)\over\g+\d},\nonumber\\ &S'(x),~I'(x),~R'(x),~S''(x),~I''(x),~R''(x)\to0.
  \end{align}
  Moreover, $S(x)$ is decreasing, $0\le I(x)\le\Sm-\Sp$ for $x\in\R$, $R(x)$ is increasing, and
  \begin{align}
    \int_{-\infty}^\infty(\g+\d) I(x)dx=\int_{-\infty}^\infty {\b S(x)I(x)\over S(x)+I(x)+R(x)}dx=c(\Sm-\Sp).
  \end{align}
\end{prop}
\begin{proof}
The existence of fixed point follows from Lemma \ref{lem-invariant}, Lemma \ref{lem-cc} and Schauder fixed point theorem.
Namely, there exists $(S,I,R)\in B_\mu(R,\R^3)$ such that
\begin{align}
  S&=F_1(S,I,R)=\D_1^{-1}[\a_1S-\b SI/(S+I+R)];\label{S-int}\\
  I&=F_2(S,I,R)=\D_2^{-1}[\a_2I+\b SI/(S+I+R)-(\g+\d) I];\label{I-int}\\
  R&=F_3(S,I,R)=\D_3^{-1}[\a_3R+\g I].\label{R-int}
\end{align}
Since $S,I,R\in C_{-\mu,\mu}(\R)$ and $\l_i^-<-\mu<\mu<\l_i^+$ for any $i=1,2,3$, it is readily seen from (\ref{D-D}) in Lemma \ref{lem-D-D} that
\begin{align*}
  \D_1S&=\a_1S-\b SI/(S+I+R);\\
  \D_2I&=\a_2I+\b SI/(S+I+R)-(\g+\d) I;\\
  \D_3R&=\a_3R+\g I.
\end{align*}
Recalling the definition of $\D_i$ (with $i=1,2,3$) in (\ref{Di}), we conclude that $(S,I,R)$ satisfies the equations (\ref{S-ODE}), (\ref{I-ODE}) and (\ref{R-ODE}).
Since $S_-\le S\le S_+$, $I_-\le I\le I_+$ and $R_-\le R\le R_+$, we obtain from the definitions of $S_\pm$, $I_\pm$ and $R_\pm$ in (\ref{S+}-\ref{R-}) and the squeeze theorem that $S(x)\to S_{-\infty}$, $I(x)\sim e^{\l_0x}$ and $R(x)\sim {\g}e^{\l_0x}/(c\l_0-d_3\l_0^2)$ as $x\to-\infty$.
Furthermore, recall the integral representation (\ref{Di-'}) for the first derivative of $\D_i^{-1}h$:
$$(\D_i^{-1}h)'(x)={\l_i^-\over\r_i}\int_{-\infty}^xe^{\l_i^-(x-y)}h(y)dy+{\l_i^+\over\r_i}\int_x^\infty e^{\l_i^+(x-y)}h(y)dy$$
for any $h\in C_{-\mu,\mu}(\R)$.
We obtain from (\ref{S-int}), (\ref{I-int}), (\ref{R-int}) and L'H\^opital's rule that
$S'(x)\to0$, $I'(x)\to0$ and $R'(x)\to0$ as $x\to-\infty$.
Finally, from (\ref{S-ODE}), (\ref{I-ODE}) and (\ref{R-ODE}), it follows that the second derivatives $S''$, $I''$ and $R''$ also vanish at $-\infty$.
This gives (\ref{asymp-}).

Now we investigate asymptotic behavior of $S$, $I$ and $R$ as $x\to\infty$.
An integration of (\ref{S-ODE}) from $-\infty$ to $x$ gives
$$d_1S'(x)=c[S(x)-S_{-\infty}]+\int_{-\infty}^x {\b S(y)I(y)\over S(y)+I(y)+R(y)}dy.$$
Since $S(x)$ is uniformly bounded, the integral on the right-hand side should be uniformly bounded; otherwise $S'(x)\to\infty$ as $x\to\infty$, which implies $S(x)\to\infty$ as $x\to\infty$, a contradiction.
Thus, we obtain integrability of $\b SI/(S+I+R)$ on $\R$, which together with the above equality implies that $S'$ is uniformly bounded on $\R$.
Note from (\ref{S-ODE}) that
$$(e^{-cx/d_1}S')'=e^{-cx/d_1}(S''-cS'/d_1)=e^{-cx/d_1}\b SI/(S+I+R)/d_1.$$
Integrating the above equality from $x$ to infinity gives
$$e^{-cx/d_1}S'(x)=-\int_x^\infty e^{-cy/d_1}{\b S(y)I(y)\over d_1[S(y)+I(y)+R(y)]}dy.$$
Hence, $S$ is non-increasing. Furthermore, since $S$ and $I$ are non-trivial; see (\ref{asymp-}), the integral on the right-hand side of the above equality can not be identically zero, which implies $S'(x)<0$ and $S(\infty)=S_\infty<S_{-\infty}$.
We are now ready to study asymptotic behavior of $I(x)$ as $x\to\infty$.
From (\ref{I-ODE}), $I(-\infty)=0$ and $I(x)\le I_+(x)=e^{\l_0x}$, we have
\begin{equation}\label{I-int'}
  I(x)={1\over\r}\int_{-\infty}^xe^{\l^-(x-y)}{\b S(y)I(y)\over S(y)+I(y)+R(y)}dy+{1\over\r}\int_x^\infty e^{\l^+(x-y)}{\b S(y)I(y)\over S(y)+I(y)+R(y)}dy,
\end{equation}
where
$$\l^\pm:={c\pm\sqrt{c^2+4d_2(\g+\d)}\over2d_2}$$
and
$$\r:=d_2(\l^+-\l^-)=\sqrt{c^2+4d_2(\g+\d)}.$$
Remark that $\l^-<0<\l_0<\l^+$ and $\l^\pm$ are the two roots of following equation
$$-d_2\l^2+c\l+\g+\d=0.$$
We would also like to mention that the integral in (\ref{I-int'}) is well defined because of Lebesgue's dominated convergence theorem and uniform boundedness of $\b SI/(S+I+R)$.
Since $\b SI/(S+I+R)$ is integrable on $\R$, it follows from the integral equation (\ref{I-int'}) and Fubini's theorem that $I$ is also integrable on $\R$, and
\begin{equation}\label{I-integral}
  \int_{-\infty}^\infty I(x)dx={1\over\g+\d}\int_{-\infty}^\infty {\b S(x)I(x)\over S(x)+I(x)+R(x)}dx.
\end{equation}
Furthermore, since
$$I'(x)={\l^-\over\r}\int_{-\infty}^xe^{\l^-(x-y)}{\b S(y)I(y)\over S(y)+I(y)+R(y)}dy+{\l^+\over\r}\int_x^\infty e^{\l^+(x-y)}{\b S(y)I(y)\over S(y)+I(y)+R(y)}dy,$$
we have from $\l^-<0<\l^+$, $\b SI/(S+I+R)\le\b I$ and $\r=d_2(\l^+-\l_-)$ that
$$|I'(x)|\le{\b\over d_2}\int_{-\infty}^\infty I(x)dx.$$
Since $I'$ is uniformly bounded and $I\ge0$ is integrable on $\R$, it is easily seen that $I(x)\to0$ as $x\to\infty$; otherwise, we can find a number $\ep>0$, a sequence $x_n\to\infty$ and a number $\k>0$ (since $I'$ is uniformly bounded) such that
$I(x)>\ep$ for all $|x-x_n|<\k$, which contradicts the integrability of $I$ on $\R$.
By integrating (\ref{I-ODE}) on the real line, it then follows from (\ref{asymp-}) and (\ref{I-integral}) that
$I'(x)\to0$ as $x\to\infty$ (noting that this can be also obtained from the integral representation of $I'$ and L'H\^opital's rule). Again, from (\ref{I-ODE}) we obtain $I''(x)\to0$ as $x\to\infty$.
Since $\b SI/(S+I+R)$ is integrable on the real line, it is readily seen from (\ref{S-ODE}) and (\ref{asymp-}) that $S'$ is uniformly bounded, which in turn implies $S''$ is also uniformly bounded. Since $S'\le0$ is integrable on $\R$, it can be shown that $S'(x)\to0$ as $x\to\infty$.
This, together with (\ref{S-ODE}) gives $S''(x)\to0$ as $x\to\infty$.
Moreover, an integration of (\ref{S-ODE}) on the real line yields
\begin{equation}\label{S-integral}
  \int_{-\infty}^\infty {\b S(x)I(x)\over S(x)+I(x)+R(x)}dx=c(\Sm-\Sp).
\end{equation}
Solving the linear equation (\ref{R-ODE}) gives
$$R(x)={\g\over c}\int_0^x I(y)dy+{\g\over c}\int_x^0 e^{(c/d_3)(x-y)}I(y)dy+C_0+C_1e^{(c/d_3)x},$$
where $C_0$ and $C_1$ are constants of integration. Substituting $x$ by $-\infty$, we obtain from $R(-\infty)=0$ that
$$C_0={\g\over c}\int_{-\infty}^0 I(y)dy.$$
Furthermore, since $R(x)\le R_+(x)=\g e^{\l_0x}/(c\l_0-d_3\l_0^2)$ and $\l_0<c/d_3$, we have $e^{-(c/d_3)x}R(x)\to0$ as $x\to\infty$.
Hence, it is readily seen that
$$C_1={\g\over c}\int_0^\infty e^{-(c/d_3)y}I(y)dy.$$
Therefore, we obtain
\begin{align}\label{R-integral}
  R(x)={\g\over c}\int_{-\infty}^x I(y)dy+{\g\over c}\int_x^\infty e^{(c/d_3)(x-y)}I(y)dy.
\end{align}
It follows from (\ref{I-integral}), (\ref{asymp+}), (\ref{S-integral}) and L'H\^opital's rule that
$$\lim_{x\to\infty}R(x)={\g\over c}\int_{-\infty}^\infty I(x)dx={\g\over\g+\d}(\Sm-\Sp).$$
Moreover, differentiating (\ref{R-integral}) once yields
$$R'(x)={\g\over d_3}\int_x^\infty e^{(c/d_3)(x-y)}I(y)dy>0.$$
Note that $I(\infty)=0$, we obtain from L'H\^opital's rule that
$$\lim_{x\to\infty}R'(x)=0.$$
Consequently, it follows from (\ref{R-ODE}) and $I(\infty)=0$ that $R''(x)\to0$ as $x\to\infty$. This proves (\ref{asymp+}).

Finally, we intend to prove the inequality $I(x)\le \Sm-\Sp$ for all $x\in\R$.
Since $I(x)\sim e^{\l_0x}$ as $x\to-\infty$ and $I(x)\to0$ as $x\to\infty$, we can define
\begin{equation}\label{J}
  J(x):=I(x)+{\g+\d\over c}\int_{-\infty}^x I(y)dy+{\g+\d\over c}\int_x^\infty e^{(c/d_2)(x-y)}I(y)dy.
\end{equation}
It follows from (\ref{asymp-}), (\ref{asymp+}), (\ref{I-integral}), (\ref{S-integral}) and L'H\^opital's rule that
$$\lim_{x\to-\infty}J(x)=0,~\lim_{x\to\infty}J(x)={\g+\d\over c}\int_{-\infty}^\infty I(x)dx=\Sm-\Sp.$$
Similarly, by differentiating (\ref{J}) once, we obtain from the asymptotic formulas (\ref{asymp-}-\ref{asymp+}) and L'H\^opital's rule that
$$J'(x)=I'(x)+{\g+\d\over d_2}\int_x^\infty e^{(c/d_2)(x-y)}I(y)dy$$
and
$$\lim_{x\to-\infty}J'(x)=0,~\lim_{x\to\infty}J'(x)=0.$$
Furthermore, by differentiating (\ref{J}) twice, it is readily seen from the differential equation for $I$ in (\ref{I-ODE}) that
$$-d_2J''+cJ'=-d_2I''+cI'+(\g+\d) I=\b SI/(S+I+R).$$
An integration of the above equation from $x$ to $\infty$ gives
$$J'(x)={1\over d_2}\int_x^\infty e^{(c/d_2)(x-y)}{\b S(y)I(y)\over S(y)+I(y)+R(y)}dy>0.$$
Here we have used the fact that $J'(\infty)=0$.
Since $J(\infty)=\Sm-\Sp$, we obtain from the above inequality that $J(x)\le \Sm-\Sp$ for all $x\in\R$. Since $I(x)\le J(x)$ by definition (\ref{J}), it follows that $I(x)\le \Sm-\Sp$ for all $x\in\R$.
This ends the proof.
\end{proof}

\section{Non-existence of traveling wave solution}
It is easily seen that the traveling wave solution $(S,I,R)$ (if exists) of (\ref{S-ODE}-\ref{R-ODE}) satisfies the following integral equation (noting that $I(\pm\infty)=0$)
\begin{equation}\label{I-int2}
  I(x)={1\over\r}\int_{-\infty}^xe^{\l^-(x-y)}{\b S(y)I(y)\over S(y)+I(y)+R(y)}dy+{1\over\r}\int_x^\infty e^{\l^+(x-y)}{\b S(y)I(y)\over S(y)+I(y)+R(y)}dy,
\end{equation}
where
$$\l^\pm:={c\pm\sqrt{c^2+4d_2(\g+\d)}\over2d_2}$$
and
$$\r:=d_2(\l^+-\l^-)=\sqrt{c^2+4d_2(\g+\d)}.$$
Remark that $\l^-<0<\l^+$ and $\l^\pm$ are the two roots of following equation
$$-d_2\l^2+c\l+\g+\d=0.$$
Note that the integral in (\ref{I-int2}) is well defined because $\b SI/(S+I+R)$ vanishes at infinity.
By (\ref{I-int2}), the derivative of $I$ has the following integral representation:
\begin{equation*}\label{I'-int2}
  I'(x)={\l^-\over\r}\int_{-\infty}^xe^{\l^-(x-y)}{\b S(y)I(y)\over S(y)+I(y)+R(y)}dy+{\l^+\over\r}\int_x^\infty e^{\l^+(x-y)}{\b S(y)I(y)\over S(y)+I(y)+R(y)}dy.
\end{equation*}
An application of L'H\^opital's rule to the above equation yields
$I'(\pm\infty)=0$.
Applying this to (\ref{I-ODE}) gives $I''(\pm\infty)=0$.
We list the asymptotic behavior of $I$ as below.
\begin{equation}\label{I-asymp}
  I(\pm\infty)=0,~ I'(\pm\infty)=0,~ I''(\pm\infty)=0.
\end{equation}
The following two propositions give the second statement in our main theorem.
\begin{prop}
  If $R_0:=\b/(\g+\d)>1$ and $c<c^*:=2\sqrt{d_2(\b-\g-\d)}$, then there does not exist a non-trivial and non-negative traveling wave solution of (\ref{S-ODE}), (\ref{I-ODE}) and (\ref{R-ODE}) such that
  $S(-\infty)=\Sm$, $S(\infty)<\Sm$, $I(\pm\infty)=0$ and $R(-\infty)=0$.
\end{prop}
\begin{proof}
We prove the statement by contradiction.
Let $(S,I,R)$ be a solution to (\ref{S-ODE}), (\ref{I-ODE}) and (\ref{R-ODE}).
Based on the argument at the beginning of this section, we have the asymptotic behavior of $I$ as listed in (\ref{I-asymp}).
Since $\b S(x)/[S(x)+I(x)+R(x)]\to \b$ as $x\to-\infty$, there exists a number $\bar x$ such that
$$\b S(x)/[S(x)+I(x)+R(x)]-\g-\d>\s:=(\b-\g-\d)/2>0$$
for all $x<\bar x$. Applying this to (\ref{I-ODE}) yields
\begin{equation}\label{I-ine2}
  cI'(x)-d_2I''(x)>\s I(x)\ge0
\end{equation}
for all $x<\bar x$. Since $cI(x)-d_2I'(x)$ is bounded as $x\to-\infty$ by (\ref{I-asymp}), it follows that $cI'(x)-d_2I''(x)$ is integrable at $-\infty$. Lebesgue's dominated convergence theorem and the above inequality implies that $I(x)$ is also integrable at $-\infty$. Define
$$K(x):=\int_{-\infty}^x I(y)dy.$$
An integration of (\ref{I-ine2}) yields
$$\s K(x)\le c I(x)-d_2I'(x)$$
for all $x<\bar x$.
A further integration of the above inequality, together with non-negativeness of $I$ gives
$$\int_{-\infty}^x K(y)dy\le (c/\s) K(x)$$
for all $x<\bar x$. Since $K$ is non-decreasing, we have
$$\eta K(x-\eta)\le\int_{x-\eta}^xK(y)dy\le (c/\s) K(x)$$
for all $\eta>0$ and all $x<\bar x$. Hence, there exists a large $\eta>0$ such that
$$K(x-\eta)<K(x)/2$$
for all $x<\bar x$.
Denote $\mu_0:=(\ln2)/\eta>0$ and let
$$L(x):=e^{-\mu_0x}K(x).$$
It follows that $$L(x-\eta)<L(x)$$
for all $x<\bar x$, which implies $L(x)=e^{-\mu_0x}K(x)$ is bounded as $x\to-\infty$.
On account of (\ref{I-asymp}), it follows from (\ref{I-ine2}) that
$$cI'(x)>d_2I''(x),~ cI(x)>d_2I'(x),~ cK(x)>d_2I(x).$$
Hence, we conclude that $e^{-\mu_0x}I(x)$, $e^{-\mu_0x}I'(x)$ and $e^{-\mu_0x}I''(x)$ are all bounded as $x\to-\infty$.
In view of (\ref{I-asymp}), they are actually uniformly bounded on the whole real line.
Moreover, since $I(x)/[S(x)+I(x)+R(x)]\le1$ and $S(x)+I(x)+R(x)\to\Sm$ as $x\to-\infty$, $e^{-\mu_0x}I(x)/[S(x)+I(x)+R(x)]$ is also uniformly bounded on $\R$.
Noting that $R(-\infty)=0$, we solve the linear equation (\ref{R-ODE}) and obtain
$$R(x)={\g\over c}\int_{-\infty}^xI(y)dy+{\g\over c}\int_x^0 e^{(c/d_3)(x-y)}I(y)dy+C_1e^{(c/d_3)x},$$
where $C_1$ is a constant of integration.
Note that $e^{-\mu_0x}I(x)$ is uniformly bounded as $x\to-\infty$. By choosing $\mu_1>0$ such that $\mu_1<\min\{\mu_0,c/d_3\}$, we have for any $x<0$,
\begin{align*}
  e^{-\mu_1x}R(x)&={\g\over c}\int_{-\infty}^xe^{-\mu_1(x-y)}e^{-\mu_1y}I(y)dy
  +{\g\over c}\int_x^0e^{(c/d_3-\mu_1)(x-y)}e^{-\mu_1y}I(y)dy+C_1e^{(c/d_3-\mu_1)x}
  \\&\le{\g\over c}\int_{-\infty}^xe^{-\mu_1y}I(y)dy
  +{\g\over c}\int_x^0e^{-\mu_1y}I(y)dy+C_1
  \\&={\g\over c}\int_{-\infty}^0e^{-\mu_1y}I(y)dy+C_1.
\end{align*}
Since $e^{-\mu_0x}I(x)$ is uniformly bounded as $x\to-\infty$ and $\mu_1<\mu_0$, it follows from the above inequality that $e^{-\mu_1x}R(x)$ is uniformly bounded as $x\to-\infty$. Therefore, $e^{-\mu_1x}R(x)/[S(x)+I(x)+R(x)]$ is uniformly bounded on $\R$.

Now, we can introduce two-side Laplace transform on the equation (\ref{I-ODE}):
$$f(\mu)\int_{-\infty}^\infty e^{-\mu x}I(x)dx=-\int_{-\infty}^\infty e^{-\mu x}I(x){\b [I(x)+R(x)]\over S(x)+I(x)+R(x)}dx,$$
where $f$ is the characteristic function defined in (\ref{f}).
The integrals on both side of the above equality are well defined for any $\mu\in(0,\mu_0)$.
Since $e^{-\mu_1x}R(x)/[S(x)+I(x)+R(x)]$ and $e^{-\mu_0x}I(x)/[S(x)+I(x)+R(x)]$ are uniformly bounded on the real line and $f(\mu)$ is always negative for all $\mu\in\R$ (noting that $c<c^*=2\sqrt{4d_2(\b-\g-\d)}$), the two Laplace integrals can be analytically continued to the whole right half plane; otherwise the integral on the left has a singularity at $\mu=\mu^*\in\R$ and it is analytic for all $\mu<\mu^*$ (cf. \cite{CC04,WW10,Wi41}).
However, since $e^{-\mu_1x}[I(x)+R(x)]/[S(x)+I(x)+R(x)]$ is uniformly bounded, the integral on the right is actually analytic for all $\mu<\mu^*+\mu_1$, a contradiction.
Thus, the above equality holds for all $\mu>0$ and can be rewritten as
$$\int_{-\infty}^\infty e^{-\mu x}I(x)\{f(\mu)+{\b [I(x)+R(x)]\over S(x)+I(x)+R(x)}\}dx=0.$$
This again leads to a contradiction because $f(\mu)+\b[I(x)+R(x)]/[S(x)+I(x)+R(x)]\to-\infty$ as $\mu\to\infty$, but $e^{-\mu x}I(x)$ is always non-negative for all $\mu\in\R$; see \cite{CC04,WW10} for early ideas in different settings.
Thus, we conclude the proof.
\end{proof}

\begin{prop}
  If $R_0:=\b/(\g+\d)\le1$, then there does not exist a non-trivial and non-negative traveling wave solution of (\ref{S-ODE}), (\ref{I-ODE}) and (\ref{R-ODE}) such that
  $S(-\infty)=\Sm$, $S(\infty)<\Sm$, $I(\pm\infty)=0$ and $R(-\infty)=0$.
\end{prop}
\begin{proof}
Again, we prove by contradiction.
Let $(S,I,R)$ be a solution to (\ref{S-ODE}), (\ref{I-ODE}) and (\ref{R-ODE}).
Based on the argument at the beginning of this section, we have the asymptotic behavior of $I$ as listed in (\ref{I-asymp}).
If $R_0:=\b/(\g+\d)\le1$, then ${\b S(x)I(x)\over S(x)+I(x)+R(x)}\le(\g+\d) I(x)$ for all $x\in\R$.
From (\ref{I-ODE}) we have
$${d\over dx}[e^{-(c/d_2)x}{d\over dx}I(x)]=-{1\over d_2}e^{-(c/d_2)x}[{\b S(x)I(x)\over S(x)+I(x)+R(x)}-(\g+\d) I(x)]\ge0,$$
which implies that the function $e^{-(c/d_2)x}I'(x)$ is non-decreasing. Since $I'(\infty)=0$ by (\ref{I-asymp}) and $e^{-(c/d_2)x}\to0$ as $x\to\infty$, it follows that $I'(x)\le0$ for all $x\in\R$. Again from $I(\pm\infty)=0$ in (\ref{I-asymp}) we obtain $I(x)=0$ for all $x\in\R$, a contradiction.
\end{proof}
One may prove the nonexistence results in a different way by analyzing the Jacobian matrix of the six-dimensional first-order linearized system of the traveling wave equation for $(S,I,R,S',I',R')$ at the equilibrium with $S=S_{-\infty}$ and $I=R=S'=I'=R'=0$. However, we prefer to use the technique of Laplace transformation which could be extended to the study of high-dimensional systems with delays.

\section{Discussion}

Broadly speaking, there are three main types of interspecific interaction: predator-prey, competition, and mutualism. Often mutualism gives rise to cooperative systems whose dynamics is better understood. In particular, the works by Lui \cite{Lui1989} and  Weinberger, Lewis and  Li \cite{Weinberger2002,Weinberger2002-1,Weinberger2007} assure that the spreading speeds of cooperative systems can be determined by their linearizations for cooperative systems. Such a phenomena is also called the linear conjecture. There have been mathematical results when the linear conjecture is not true.   Hadeler and Rothe \cite{Hadeler1975}  showed that linear determinacy can be violated in certain cases;  the linear
conjecture is not always true.  Hosono \cite{Hosnono1998} showed numerically that the
linear determinacy may not be true for the Lotka-Volterra competition model.
W. Huang and M. Han \cite{Huang2011} showed analytically the linear determinacy does not hold for some range
of parameters.

It is known that some competition models can be converted into cooperative systems and therefore are linearly determinant with some appropriate assumptions \cite{LinLiRuan2011,WangCCC2010,Weinberger2002}. The interaction of predator-prey
describes that the predator species benefits from killing and consuming the prey species, and the
prey population size may be regulated as a result, which can also be often used in models of SIR models for infectious disease and plant-herbivore interactions. Predator-prey models are more difficult to study and remain a challenging to the mathematical biology community \cite{CCC2013,Huang2003}.

Our key contribution in this paper is that we are able to show the minimum speed for the general diffusive model (\ref{S-PDEN}-\ref{R-PDEN}) with non-constant total population can be determined by its linearization at the initial disease-free equilibrium. Our result may contribute to shed a light on the linear conjecture for predator-prey systems. As discussed in \cite{WWW12}, it is reasonable to believe that the existence of traveling waves solution is valid for $c=c^*$. Our method in this paper is mainly based on the Schauder fixed point theorem for equivalent non-monotone abstract operators. Similar ideas have been used in \cite{WWW12,WW10}.    Related methods  can also be found in other studies (e.g. \cite{LinLiRuan2011,Ma01,Wa09,WLW12} and other recent papers for different types of spatial models, also see a brief survey at the end of \cite{WWW12}). However, the general diffusive Kermack-McKendrick SIR model (\ref{S-PDEN}-\ref{R-PDEN}) involves three unknown variables $S, I, R$, several significant new ingredients have been introduced in the proof.  Specifically,  one of the challenging and difficult tasks in this paper is to construct and verify a suitable invariant convex set of three dimensions ($S, I, R$) for the non-monotone operators.  The approach in this paper provides a promising method to deal with high dimensional epidemic models.  It would be difficult, if possible, to investigate traveling waves for high dimensional models with phase portraits analysis.

One might attempt to solve for $R$ in term of some integral of $I$ and then reduce (\ref{S-ODE}-\ref{R-ODE}) to a system only with $S,I$. It is unclear such an approach can be used to prove the existence of traveling waves.  However, it would make the SIR model more complicated and prevent seeing biological meanings of $R$. The construction of a three dimensional invariant set is a better option for this problem as our results involve nonexistence of traveling waves, more importantly, we try to provide an effective approach to deal with traveling waves of high dimensional SIR models.

\section{Appendix}
\subsection{Proof of Lemma \ref{lem-sub-ine}}
\begin{proof}
  In view of (\ref{S-}) and (\ref{I+}), the first inequality (\ref{S-ine}) is the same as
  $$-\b e^{\l_0x}\ge S_{-\infty}e^{\ep_1x}(d_1M_1\ep_1^2-cM_1\ep_1)$$
  for all $x\le x_1:=-\ep_1^{-1}\ln M_1$. The above inequality can be written as
  $$M_1\ep_1(c-d_1\ep_1)\ge\b e^{(\l_0-\ep_1)x}.$$
  Note that $x\le x_1:=-\ep_1^{-1}\ln M_1$. It suffices to prove
  $$M_1\ep_1(c-d_1\ep_1)\ge\b M_1^{-(\l_0-\ep_1)/\ep_1},$$
  which is obviously true if we choose $\ep_1>0$ such that $\ep_1<\min\{\l_0,c/d_1\}$ and then let $M_1$ be sufficiently large.

  Now we intend to prove the second inequality (\ref{I-ine}),
  which, by subtracting both sides by $(\b-\g-\d)I_-$, is the same as
  $$-{\b I_-(I_++R_+)\over S_-+I_++R_+}\ge-d_2I_-''+cI_-'-(\b-\g-\d)I_-=-M_2f(\l_0+\ep_2)e^{(\l_0+\ep_2)x},$$
  where $f$ is defined in (\ref{f}) with $\l_0$ as its smaller root.
  Note that $f$ is concave down, we can choose a sufficiently small $\ep_2\in(0,\ep_1)$ such that $f(\l_0+\ep_2)>0$.
  Then, we assume $M_2$ is sufficiently large that $x_2<x_1$ holds.
  It suffices to show
  $$M_2f(\l_0+\ep_2)e^{(\l_0+\ep_2)x}\ge {\b I_+(I_++R_+)\over S_-},$$
  which, in view of (\ref{S-}), (\ref{I+}) and (\ref{R+}), is equivalent to
  $$M_2f(\l_0+\ep_2)\Sm(1-M_1e^{\ep_1x})\ge {\b(\g+c\l_0-d_3\l_0^2)\over c\l_0-d_3\l_0^2}e^{(\l_0-\ep_2)x}.$$
  Noting that $x\le x_2:=-\ep_2^{-1}\ln M_2$, we only need to prove
  $$M_2f(\l_0+\ep_2)\Sm(1-M_1M_2^{-\ep_1/\ep_2})\ge {\b(\g+c\l_0-d_3\l_0^2)\over c\l_0-d_3\l_0^2}M_2^{-(\l_0-\ep_2)/\ep_2},$$
  which is true for large $M_2$ because as $M_2\to\infty$, the left-hand side tends to infinity and the right-hand side vanishes (recall that $0<\ep_2<\ep_1<\l_0$).

  Finally, we are ready to verify the last inequality (\ref{R-ine}). First, since $c-d_3\l_0>0$ in (\ref{d3-ine}), we can choose $\ep_3\in(0,\ep_2)$ so small that
  $c-d_3(\l_0+\ep_3)>0$. In view of (\ref{I-}) and (\ref{R-}), the inequality (\ref{R-ine}) can be written as
  $$\g e^{\l_0x}(1-M_2e^{\ep_2x})\ge{\g\over c\l_0-d_3\l_0^2}\{e^{\l_0x}(c\l_0-d_3\l_0^2)-M_3e^{(\l_0+\ep_3)x}[c(\l_0+\ep_3)-d_3(\l_0+\ep_3)^2]\},$$
  which is equivalent to
  $${c(\l_0+\ep_3)-d_3(\l_0+\ep_3)^2\over c\l_0-d_3\l_0^2}M_3\ge M_2e^{(\ep_2-\ep_3)x}.$$
  Note that $\ep_3<\ep_2$ and $x\le x_3:=-\ep_3^{-1}\ln M_3$, it suffices to prove the above inequality for $x=x_3$:
  $${c(\l_0+\ep_3)-d_3(\l_0+\ep_3)^2\over c\l_0-d_3\l_0^2}M_3\ge M_2M_3^{-(\ep_2-\ep_3)/\ep_3}.$$
  This is true for large $M_3$ because as $M_3\to\infty$, the left-hand side tends to infinity and the right-hand side vanishes.
  This ends the proof of our lemma.
\end{proof}

\subsection{Proof of Lemma \ref{lem-D-D}}
\begin{proof}
  It follows from the definitions of $\D_i$ and $\D_i^{-1}$ in (\ref{Di}) and (\ref{Di-}) that
  \begin{eqnarray*}
    [\D_i^{-1}(\D_ih)](x)&=&{1\over\r_i}\int_{-\infty}^xe^{\l_i^-(x-y)}[-d_ih''(y)+ch'(y)+\a_ih(y)]dy
    \\&&+{1\over\r_i}\int_x^\infty e^{\l_i^+(x-y)}[-d_ih''(y)+ch'(y)+\a_ih(y)]dy.
  \end{eqnarray*}
  Making use of integration by parts, we obtain
  $$\int_{-\infty}^xe^{\l_i^-(x-y)}h'(y)dy=h(x)+\l_i^-\int_{-\infty}^xe^{\l_i^-(x-y)}h(y)dy,$$
  and
  $$\int_{-\infty}^xe^{\l_i^-(x-y)}h''(y)dy=h'(x)+\l_i^-h(x)+(\l_i^-)^2\int_{-\infty}^xe^{\l_i^-(x-y)}h(y)dy.$$
  Therefore, we have
  $$\int_{-\infty}^xe^{\l_i^-(x-y)}[-d_ih''(y)+ch'(y)+\a_ih(y)]dy
  =-d_ih'(x)+(-d_i\l_i^-+c)h(x).$$
  Here we have used the fact that $\l_i^-$ is a root of the function $f_i(\l)=-d_i\l^2+c\l+\a_i$; see (\ref{li}) and (\ref{fi}).
  Similarly, it can be shown that
  $$\int_x^\infty e^{\l_i^+(x-y)}[-d_ih''(y)+ch'(y)+\a_ih(y)]dy=d_ih'(x)+(d_i\l_i^+-c)h(x).$$
  Applying the above two equalities to the expression of $\D_i^{-1}(\D_ih)$ gives
  $$[\D_i^{-1}(\D_ih)](x)={d_i(\l_i^+-\l_i^-)\over\r_i}h(x)=h(x),$$
  where in the last equality we have used the definition of $\r_i$ in (\ref{ri}).
  This proves (\ref{D-D}).

   Let $x^*:=-\ln M/\ep$ be the point where $g$ is not differentiable.
  Recall from (\ref{li}) and (\ref{fi}) that
  \begin{eqnarray}\label{fi2}
    f_i(k):=-d_ik^2+ck+\a_i=d_i(k-\l_i^-)(\l_i^+-k).
  \end{eqnarray}
  It is easily seen from (\ref{Di}) that
  \begin{equation}\label{Dig}
    (\D_ig)(x)=\begin{cases}
      f_i(\l)e^{\l x}-Mf_i(\l+\ep)e^{(\l+\ep)x},&\ x<x^*,\\
      0,&\ x>x^*.
    \end{cases}
  \end{equation}
  To prove (\ref{D-D-ine}), we will consider the two cases $x\le x^*$ and $x\ge x^*$ respectively.
  When $x\le x^*$, we have from (\ref{Di-}) and (\ref{Dig}) that
  \begin{eqnarray}\label{D-Dg1}
    [\D_i^{-1}(\D_ig)](x)=f_i(\l)A(\l)-Mf_i(\l+\ep)A(\l+\ep),
  \end{eqnarray}
  where
  \begin{eqnarray*}
    A(k)&:=&{1\over\r_i}\int_{-\infty}^xe^{\l_i^-(x-y)+ky}dy+{1\over\r_i}\int_x^{x^*}e^{\l_i^+(x-y)+ky}dy
    \\&=&{e^{kx}(\l_i^+-\l_i^-)\over\r_i(k-\l_i^-)(\l_i^+-k)}-{e^{kx^*+\l_i^+(x-x^*)}\over\r_i(\l_i^+-k)}
  \end{eqnarray*}
  for $k=\l$ or $\l+\ep$.
  In view of (\ref{ri}) and (\ref{fi2}), it follows from the above equality that
  $$f_i(k)A(k)=e^{kx}-{k-\l_i^-\over\l_i^+-\l_i^-}e^{kx^*+\l_i^+(x-x^*)}.$$
  Applying this to (\ref{D-Dg1}) and on account of $Me^{\ep x^*}=1$, we obtain
  \begin{eqnarray*}
    [\D_i^{-1}(\D_ig)](x)&=&[e^{\l x}-{\l-\l_i^-\over\l_i^+-\l_i^-}e^{\l x^*+\l_i^+(x-x^*)}]
    \\&&-[Me^{(\l+\ep) x}-{\l+\ep-\l_i^-\over\l_i^+-\l_i^-}e^{\l x^*+\l_i^+(x-x^*)}]
    \\&=&[e^{\l x}-Me^{(\l+\ep) x}]+{\ep\over\l_i^+-\l_i^-}e^{\l x^*+\l_i^+(x-x^*)}
    \\&\ge&e^{\l x}-Me^{(\l+\ep) x}.
  \end{eqnarray*}
  This proves (\ref{D-D-ine}) for $x\le x^*$. When $x\ge x^*$, we have from (\ref{Di-}) and (\ref{Dig}) that
  \begin{eqnarray}\label{D-Dg2}
    [\D_i^{-1}(\D_ig)](x)=f_i(\l)B(\l)-Mf_i(\l+\ep)B(\l+\ep),
  \end{eqnarray}
  where
  \begin{eqnarray*}
    B(k):={1\over\r_i}\int_{-\infty}^{x^*}e^{\l_i^-(x-y)+ky}dy
    ={e^{kx^*+\l_i^-(x-x^*)}\over\r_i(k-\l_i^-)}
  \end{eqnarray*}
  for $k=\l$ or $\l+\ep$.
  In view of (\ref{ri}) and (\ref{fi2}), it follows from the above equality that
  $$f_i(k)B(k)={\l_i^+-k\over\l_i^+-\l_i^-}e^{kx^*+\l_i^-(x-x^*)}.$$
  Applying this to (\ref{D-Dg2}) and on account of $Me^{\ep x^*}=1$, we obtain
  \begin{eqnarray*}
    [\D_i^{-1}(\D_ig)](x)&=&{\l_i^+-\l\over\l_i^+-\l_i^-}e^{\l x^*+\l_i^-(x-x^*)}
    -{\l_i^+-\l-\ep\over\l_i^+-\l_i^-}e^{\l x^*+\l_i^-(x-x^*)}
    \\&=&{\ep\over\l_i^+-\l_i^-}e^{\l x^*+\l_i^-(x-x^*)}
    \\&\ge&0.
  \end{eqnarray*}
  This gives (\ref{D-D-ine}) in the case $x\ge x^*$.
\end{proof}

\subsection{Proof of Lemma \ref{lem-invariant}}
\begin{proof}
  Throughout this proof, we will frequently use the inequalities $0<S/(S+I+R)<1$ and $0<I/(S+I+R)<1$.
  Since $\a_1S-\b SI/(S+I+R)\le\a_1S_+=\D_1S_+$; see the definition of $\D_1$ in (\ref{Di}), we obtain from (\ref{F1}) and (\ref{D-D}) that
  $$F_1(S,I,R)\le\D_1^{-1}(\D_1S_+)=S_+.$$
  By (\ref{S-ine}) in Lemma \ref{lem-sub-ine}, we have for $x\le x_1$,
  $$\a_1S-\b SI/(S+I+R)\ge\a_1S_--\b I_+\ge\a_1S_--d_1S_-''+cS_-'=\D_1S_-.
  $$
  When $x\ge x_1$, it follows from $\a_1>\b$ (recalling the choice of $\a_1$ in the paragraph after (\ref{Di-''})) and $S_-(x)=0$ that
  $$\a_1S-\b SI/(S+I+R)\ge(\a_1-\b)S\ge0=\D_1S_-.$$
  Coupling the above two inequalities and making use of (\ref{D-D-ine}) yield
  $$F_1(S,I,R)\ge\D_1^{-1}(\D_1S_-)\ge S_-.$$
  Since $\a_2>\g+\d$ (by the choice of $\a_2$) and $\l_0$ is a root of $f$ defined in (\ref{f}), we have
  $$\a_2I+\b SI/(S+I+R)-(\g+\d) I\le\a_2I_++\b I_+-(\g+\d) I_+=\a_2I_+-d_2I_+''+cI_+'=\D_2I_+.$$
  In view of (\ref{D-D}), we obtain from the above inequality that
  $$F_2(S,I,R)\le\D_2^{-1}(\D_2I_+)=I_+.$$
  By (\ref{I-ine}) in Lemma \ref{lem-sub-ine} and monotonicity of $\b SI/(S+I+R)$ in $S$, we obtain
  \begin{align*}
  \a_2I+\b SI/(S+I+R)-(\g+\d) I&\ge\a_2I_-+\b S_-I_-/(S_-+I_++R_+)-(\g+\d) I_-
  \\&\ge\a_2 I_--d_2I_-''+cI_-'
  \\&=\D_2I_-
  \end{align*}
  for $x\le x_2$. When $x\ge x_2$, it is readily seen from $\a_2>\g+\d$ and $I_-(x)=0$ that
  $$\a_2I+\b SI/(S+I+R)-(\g+\d) I\ge\a_2I-(\g+\d) I\ge0=\D_2I_-.$$
  A combination of the above two inequalities and (\ref{D-D-ine}) yields
  $$F_2(S,I,R)\ge\D_2^{-1}(\D_2I_-)\ge I_-.$$
  From the definitions of $I_+$ and $R_+$ in (\ref{I+}) and (\ref{R+}), we have
  $$\a_3R+\g I\le\a_3R_++\g I_+=\a_3R_++cR_+'-d_3R_+''=\D_3R_+.$$
  Thus, it follows from (\ref{F3}) and (\ref{D-D}) that
  $$F_3(S,I,R)\le\D_3^{-1}(\D_3R_+)=R_+.$$
  When $x\le x_3$, we obtain from (\ref{R-ine}) that
  $$\a_3R+\g I\ge\a_3R_-+\g I_-\ge\a_3 R_-+cR_-'-d_3R_-''=\D_3R_-.$$
  When $x\ge x_3$, we have $R_-(x)=0$ and
  $$\\a_3R+\g I\ge0=\D_3R_-.$$
  Hence, it follows from (\ref{D-D-ine}) that
  $$F_3(S,I,R)\ge\D_3^{-1}(\D_3R_-)\ge R_-.$$
  This ends our proof of the lemma.
\end{proof}

\subsection{Proof of Lemma \ref{lem-cc}}
\begin{proof}
  Note that the standard incidence function $\b SI/(S+I+R)$ has bounded partial derivatives with respect to $S$, $I$ and $R$. For example, the partial derivative of $\b SI/(S+I+R)$ with respect to $S$ is ${\b (I+R)I/(S+I+R)^2},$ which is bounded by $\b$. Similarly, we can show that the partial derivatives with respect to $I$ and $R$ are also bounded by $\b$. Therefore, for any $(S_1,I_1,R_1)\in\G$ and $(S_2,I_2,R_2)\in\G$, we have
  $$|{\b S_1I_1\over S_1+I_1+R_1}-{\b S_2I_2\over S_2+I_2+R_2}|\le \b(|S_1-S_2|+|I_1-I_2|+|R_1-R_2|).$$
  It is readily seen that
  $$|(\a_1S_1-{\b S_1I_1\over S_1+I_1+R_1})-(\a_1S_2-{\b S_2I_2\over S_2+I_2+R_2})|
  \le(\a_1+\b)(|S_1-S_2|+|I_1-I_2|+|R_1-R_2|).$$
  Consequently, we obtain from the definition (\ref{F1}) that
  \begin{align*}
    |F_1(S_1,I_1,R_1)(x)-F_1(S_2,I_2,R_2)(x)|e^{-\mu|x|}\le{\a_1+\b\over\r_1}(|S_1-S_2|_\mu+|I_1-I_2|_\mu+|R_1-R_2|_\mu)C(x),
  \end{align*}
  where
  $$C(x):=e^{-\mu|x|}[\int_{-\infty}^x e^{\l_1^-(x-y)+\mu|y|}dy+\int_x^\infty e^{\l_1^+(x-y)+\mu|y|}dy].$$
  Here $S_1-S_2\in C_{-\mu,\mu}(\R)=B_\mu(\R,\R)$ and $|S_1-S_2|_\mu=\sup_{x\in\R}e^{-\mu|x|}|S_1(x)-S_2(x)|$; see (\ref{Bmu}) and (\ref{norm}).
  To prove the continuity of $F_1$, it suffices to show that $C(x)$ is uniformly bounded for $x\in\R$.
  Since $\l_1^-<-\mu<\mu<\l_1^+$, applying L'H\^opital's rule to the above formula yields
  $$C(-\infty)={1\over\mu+\l_1^+}-{1\over\mu+\l_1^-}$$
  and
  $$C(\infty)={1\over\l_1^+-\mu}+{1\over\mu-\l_1^-}.$$
  Hence, we conclude that $C(x)$ is uniformly bounded on $\R$ and thus $F_1$ is a continuous map from $\G$ to $B_\mu(\R,\R)$ with respect to the norm $|\cdot|_\mu$.
  Similarly, we can show that $F_2$ and $F_3$ are also continuous. Consequently, $F$ is a continuous map on $\G$ with respect to the norm $|\cdot|_\mu$.

  To prove the compactness of $F$, we shall make use of Arzela-Ascoli theorem and a standard diagonal process.
  Let $I_k:=[-k,k]$ with $k\in\N$ be a compact interval on $\R$ and temporarily we regard $\G$ as a bounded subset of $C(I_k,\R^3)$ equipped with the maximum norm.
  Since $F$ maps $\G$ into $\G$, it is obvious that $F$ is uniformly bounded.
  We will use the following two inequalities to show that $F$ is equi-continuous.
  Namely, from the definition of $F_i$ in (\ref{F1}-\ref{F2}) and integral representation for the derivative of $\D_i^{-1}$ in (\ref{Di-'}) we have for any $(S,I,R)\in\G$,
  \begin{align*}
    |[F_1(S,I,R)]'(x)|&\le{-\l_1^-\a_1S_{-\infty}\over\r_1}\int_{-\infty}^xe^{\l_1^-(x-y)}dy
    +{\l_1^+\a_1S_{-\infty}\over\r_1}\int_x^\infty e^{\l_1^+(x-y)}dy
    \\&={2\a_1S_{-\infty}\over\r_1},
  \end{align*}
  and
  \begin{align*}
    |[F_2(S,I,R)]'(x)|&\le{-\l_2^-(\a_2+\b-\g-\d)\over\r_2}\int_{-\infty}^xe^{\l_2^-(x-y)+\l_0y}dy
    \\&+{\l_2^+(\a_2+\b-\g-\d)\over\r_2}\int_x^\infty e^{\l_2^+(x-y)+\l_0y}dy
    \\&={(\a_2+\b-\g-\d)e^{\l_0 x}\over\r_2}({-\l_2^-\over\l_0-\l_2^-}+{\l_2^+\over\l_2^+-\l_0})
    \\&={c\l_0+2\a_2\over\r_2}e^{\l_0 x},
  \end{align*}
  and
  \begin{align*}
    |[F_3(S,I,R)]'(x)|&\le{-\l_3^-\g(\a_3+c\l_0-d_3\l_0^2)\over\r_3(c\l_0-d_3\l_0^2)}\int_{-\infty}^xe^{\l_3^-(x-y)+\l_0y}dy
    \\&+{\l_3^+\g(\a_3+c\l_0-d_3\l_0^2)\over\r_3(c\l_0-d_3\l_0^2)}\int_x^\infty e^{\l_3^+(x-y)+\l_0y}dy
    \\&={e^{\l_0 x}\g(\a_3+c\l_0-d_3\l_0^2)\over\r_3(c\l_0-d_3\l_0^2)}({-\l_3^-\over\l_0-\l_3^-}+{\l_3^+\over\l_3^+-\l_0})
    \\&={\g(c\l_0+2\a_3)\over\r_3(c\l_0-d_3\l_0^2)}e^{\l_0 x}
  \end{align*}
  Here we have made use of the facts that $\l_0$ defined in (\ref{l0}) is a root of $f$ in (\ref{f}) and $\l_i^\pm$ defined in (\ref{li}) are the roots of $f_i$ in (\ref{fi}).
  Let $\{u_n\}$ be a sequence of $\G$, which can be also viewed as a bounded subset of $C(I_k)$ with $I_k:=[-k,k]$. Since $F$ is uniformly bounded and equi-continuous, by the Arzela-Ascoli theorem and the standard diagonal process, we can extract a subsequence $\{u_{n_k}\}$ such that $v_{n_k}:=Fu_{n_k}$ converges in $C(I_k)$ for any $k\in\N$. Let $v$ be the limit of $v_{n_k}$.
  It is readily seen that $v\in C(\R,\R^3)$. Furthermore, since $F(\G)\subset\G$ by Lemma \ref{lem-invariant} and $\G$ is closed, it follows that $v\in\G$. Now we come back to the norm $|\cdot|_\mu$ defined in (\ref{norm}).
  Note that $\mu>\l_0>0$, it follows from (\ref{I+}) and (\ref{R+}) that $e^{-\mu |x|}I_+(x)$ and $e^{-\mu |x|}R_+(x)$ are uniformly bounded on $\R$. Thus, $\G$ is uniformly bounded with respect to the norm $|\cdot|_\mu$. Consequently, the norm $|v_{n_k}-v|_\mu$ is uniformly bounded for all $k\in\N$. Given any $\ep>0$, we can find an integer $M>0$ independent of $v_{n_k}$ such that
  $$e^{-\mu|x|}|v_{n_k}(x)-v(x)|<\ep$$ for any $|x|>M$ and $k\in\N$.
  Since $v_{n_k}$ converges to $v$ on the compact interval $[-M,M]$ with respect to the maximum norm, there exists $K\in\N$ such that
  $$e^{-\mu|x|}|v_{n_k}(x)-v(x)|<\ep$$ for any $|x|\le M$ and $k>K$.
  The above two inequalities imply that $v_{n_k}$ converges to $v$ with respect to the norm $|\cdot|_\mu$.
  This proves the compactness of the map $F$.
\end{proof}

\end{document}